\documentclass[12pt,eqright]{amsart}
\usepackage{amssymb,amsmath}
\usepackage{bm}
\usepackage{amsfonts}
\usepackage{epsf}
\usepackage{graphicx}
\usepackage{color}
\newtheorem{theorem}{Theorem}[section]
\newtheorem{lemma}[theorem]{Lemma}

\newtheorem{algorithm}[theorem]{Algorithm}
\theoremstyle{definition}

\numberwithin{equation}{section}
\allowdisplaybreaks 

\newcommand{\bc}{\begin{center}}
\newcommand{\ec}{\end{center}}
\newcommand{\be}{\begin{eqnarray}}
\newcommand{\ee}{\end{eqnarray}}

\newcommand{\ben}{\begin{eqnarray*}}
\newcommand{\een}{\end{eqnarray*}}

\title[Lower and Upper Bounds of Eigenvalue Problem]{ Computing the lower and upper bounds of Laplace eigenvalue
problem: by combining conforming and nonconforming finite element
methods}

\author[F.~Luo,  Q.~Lin, H.~Xie]
{Fusheng Luo$^\ast$, Qun Lin$^\dagger$, and Hehu Xie$^\ddag$}

\address{$^\ast$ LSEC, ICMSEC, Academy of
Mathematics and Systems Science, Chinese Academy of Sciences,
Beijing 100190, China}
\email{fshluo@lsec.cc.ac.cn}
\address{$^\dagger$ LSEC, ICMSEC, Academy of
Mathematics and Systems Science, Chinese Academy of Sciences,
Beijing 100190, China}
\email{linq@lsec.cc.ac.cn}
\address{$^\ddag$ LSEC, ICMSEC, Academy of
Mathematics and Systems Science, Chinese Academy of Sciences,
Beijing 100190, China}
\email{hhxie@lsec.cc.ac.cn}

\thanks{This work is supported in part
by the National Science Foundation of China (NSFC 11001259).
}

\date{June 25, 2011}

\begin{document}
\begin{abstract}

This article is devoted to computing the lower and upper bounds of
the Laplace eigenvalue problem. By using the special nonconforming
finite elements, i.e., enriched Crouzeix-Raviart element and
extension $Q_1^{\rm rot}$, we get the lower bound of the eigenvalue.
Additionally, we also use conforming finite elements to do the
postprocessing to get the upper bound of the eigenvalue. The
postprocessing method need only to solve the corresponding source
problems and a small eigenvalue problem if higher order
postprocessing method is implemented. Thus, we can obtain the lower
and upper bounds of the eigenvalues simultaneously by solving
eigenvalue problem only once. Some numerical results are also
presented to validate our theoretical analysis.
\end{abstract}
\maketitle

\section{Introduction}\label{S1}
The eigenvalue problems are important, which appears in many fields,
such as quantum mechanics, fluid mechanics, stochastic process and etc.
Thus, a fundamental work is to find the eigenvalues of partial
differential equations. From last century, abundant works are
dedicated to this topic.

Feng in his famous paper \cite{Feng1965} cites the pioneer work of
P\'{o}lya in computing the upper bound of Laplace eigenvalue
problem. And based on the minimum-maximal principle discovered by
Rayleigh, Poincar\'{e}, Courant and Fischer ect., any conforming
finite element method will give the upper bound (see Strang and Fix
\cite{StrangFix}). Nevertheless, to the lower bound aspect, until
1979, Rannacher \cite{Rannacher} gives some numerical results for
plate problem.

And then there is few work on analysis of the lower bound for a long
time. Hu, Huang and Shen \cite{HuHuangShen} get the lower bound of
Laplace equation by conforming linear and bilinear elements together with the
mass lumping method. Inspired by the minimum-maximal principle,
people try to find the lower bound with the nonconforming element
methods. Recently, a series of works make progress in this aspect,
e.g.  Lin and Lin \cite{LinLin2006} use the asymptotic expansion
skill to compute the eigenvalues by nonconforming finite element
method; also see the numerical reports of Liu and Liu \cite{LiuLiu},
Liu and Yan \cite{LiuYan}, and the work of Lin, Huang and Li
\cite{LinHuangLi}. Another way by Armentano and Dur\'{a}n
\cite{ArmentanoDuran} is to use another kind of expansion method to
get the lower bound, which is of less restriction to the smoothness
of eigenfunctions compared with the asymptotic expansion skill. Also
see the follow-up works by Li \cite{lya}, Lin \cite{LinXie}, Yang
\cite{YangZhangLin}.

Inspired by these works, this article propose a method to obtain the
lower and upper bounds of the eigenvalue simultaneously which only
need to solve the eigenvalue problem once and additional auxiliary
source problem. Our method can be described as follows: (1) solve
the eigenvalue problem by some nonconforming finite element; (2)
solve an additional auxiliary source problem in an conforming finite
element space. Since we obtain not only lower bound but also upper
bound of the eigenvalue, we can give the accurate error of the
eigenvalue approximations. Compared with the existed literature,
this note contributes on the following aspects:
\begin{itemize}
\item The assumption of the lower bound in
\cite{ArmentanoDuran} needs an critical assumption that:
$|u-u_h|_1\geq Ch^{\gamma}$, $\gamma<1$, which promise itself be the
dominant term in the expansion. By our recent result of lower bound of convergence rate by
finite element method, we get rid of this constrain.
\item A new application of the correction method of eigenvalue problem
is proposed to obtain the lower and upper bounds of the eigenvalues
by solving the eigenvalue problem once and an additional source
problem.
\item By using higher order conforming finite element to do the
correction in a new way, we also prove the upper bound of the
corrected eigenvalue approximations.
\item After calculating upper and lower bounds simultaneously,
we can find how much accuracy we have actually
achieved (an accurate a posteriori error estimate).
\end{itemize}

For simplicity, we only discuss the problem in $\mathcal{R}^2$, but
the methods and results here can be easily extended to the case
$\mathcal{R}^3$. In this paper, we will use the standard notation of
Sobolev spaces (\cite{StrangFix}). The outline of the paper is as
follows. In Section 2, some preliminaries and notation are
introduced. The weak form of the Laplace eigenvalue problem and
its corresponding discrete form is stated. In Section 3, we will
give the results about the lower bound with nonconforming finite
elements. Section 4 is devoted to analyzing the upper bound of the
eigenvalue by postprocessing with the lowest order conforming finite
element methods. In section 5, another type of postprocessing method
is proposed to obtain not only higher order accuracy but also upper
bound approximation of the eigenvalues. Some numerical results are
presented in section 6 to test our theoretical results and some
concluding remarks are given in the last section.

\section{The eigenvalue problem}
In this paper we are concerned with the Laplace eigenvalue problem:

Find $(\lambda,u)$ such that
\begin{equation}\left\{
\begin{array}{rcl}
 -\triangle u&=&\lambda u, \quad \text{in } \Omega,\\
 u&=&0,\quad\ \  \text{on } \partial \Omega,\\
\int_{\Omega}u^2dxdy&=&1,
 \end{array} \right.\label{laplaceproblem}
\end{equation}
where $\Omega$ is a bounded domain in $\mathcal{R}^2$ with
continuous Lipschitz  boundary $\partial\Omega$ .

The variational problem associated with (\ref{laplaceproblem}) is
given by:

Find $u\in V:=H_0^1(\Omega)$ and $\lambda \in \mathcal{R}$ such that
$b(u,u)=1$ and
\begin{equation}\label{weaklaplaceproblem}
 a(u,v)=\lambda b(u,v), \quad \forall v\in V,
\end{equation}
where \begin{eqnarray*}
a(u,v)=\int_{\Omega}\nabla u\cdot\nabla v
dxdy\ \ {\rm and}\ \ \ b(u,v)=\int_{\Omega}uvdxdy.
\end{eqnarray*}
From \cite{BA91}, we know that the Lapalce eigenvalue problem has a positive
eigenvalue sequence $\{\lambda_j\}$ with
$$0< \lambda_1\leq
\lambda_2\leq\cdots\leq \lambda_k\leq \cdots ,\quad
\lim\limits_{k\rightarrow\infty}\lambda_k=\infty,$$ and the
corresponding eigenfunction sequence $\{u_j\}$
$$u_1,u_2,\cdots,u_k,\cdots,$$ with the property
$b(u_i,u_j)=\delta_{ij}$.

From the result in \cite{BA91}, the Rayleigh quotient is defined by
$$R(u):=\frac{a(u,u)}{b(u,u)}=\lambda.$$

We define the discrete finite element method to solve the problem
(\ref{dislaplaceproblem}). $\mathcal{T}_h$ is a quasi-uniform
triangulation. Based on this partition, $\mathcal{E}_h$ denotes the
set of all edges in partition $\mathcal{T}_h$. The finite element
space $V_h$ is the corresponding finite element space to the
partition, i.e. $V^{NC}_h\nsubseteq V$ as a nonconforming space and
$V^{C}_h\subset V$ a conforming space.

In the rest of this paper, the finite element space $V_h$ can be
$V_h^C$ or $V_h^{NC}$. The finite element approximation of
(\ref{weaklaplaceproblem}) is defined as follows:

Find $(\lambda_h,u_h) \in \mathcal{R}\times V_h$ such that
$b(u_h,u_h)=1$ and
\begin{eqnarray}\label{dislaplaceproblem}
 a_h(u_h,v_h)=\lambda_h b(u_h,v_h), \quad \forall v_h\in V_h,
\end{eqnarray}
where the bilinear form $a_h(\cdot,\cdot)$ coincides with
$a(\cdot,\cdot)$ in conforming finite element, or a elementwise
representation of $a(\cdot,\cdot)$ in nonconforming situation, e.g.
\begin{align*}
a_h(u_h,v_h)=\sum\limits_{K\in\mathcal{T}_h}\int_{k}\nabla u_h\cdot\nabla
v_hdxdy.
\end{align*}
Obviously, the bilinear form for conforming situation can also be
presented as this form. It is easy to see that for both situation,
the bilinear are $V_h-$elliptic. Thus, we define the norm on $V_h+V$
by
\begin{eqnarray*}
\|v\|_{a,h}^2&=&a_h(v,v), \ \ \ \ {\rm for}\ v\in V_h^{NC},
\end{eqnarray*}
and
\begin{eqnarray*}
\|v\|_a^2&=&a(v,v),\ \ \ \ \ \ {\rm for}\ v\in
V+V_h^{C}.
\end{eqnarray*}

For both conforming and nonconforming situations, the Rayleigh
quotient holds for the eigenvalue $\lambda_h$
$$R(u_h)=\frac{a_h(u_h,u_h)}{b(u_h,u_h)}=\lambda_h.$$ Similarly, the
discrete eigenvalues problem (\ref{dislaplaceproblem}) has also an
eigenvalue sequence $\{\lambda_{j,h}\}$ with 
$$0< \lambda_{1,h}\leq \lambda_{2,h}\leq\cdots\leq \lambda_{N,h},$$
and the corresponding discrete eigenfunction sequence $\{u_{k,h}\}$
$$u_{1,h},u_{2,h},\cdots,u_{k,h},\cdots,u_{N,h}$$
with the property $b(u_{i,h},u_{j,h})=\delta_{ij}, 1\leq i,j\leq
N_h$ ($N_h$ is the dimension of $V_h$).

We cite here the classical
 result about eigenvalue: minimum-maximum principle. Let $\lambda_j$ be the $j-$th eigenvalue of
 (\ref{weaklaplaceproblem}) and $\lambda_{j,h}$ be the $j-$th eigenvalue of (\ref{dislaplaceproblem}),
 respectively. Arranging them by increasing order, then we have (\cite{BA91})
\begin{eqnarray}\label{min_Max_principle}
\lambda_j=\min\limits_{V_j\subset V,\dim V_j=j}\max\limits_{v\in
V_j}R(v),\quad \lambda_{j,h} =\min\limits_{V_j\subset V_h, \dim
V_j=j}\max\limits_{v\in V_j}R(v).
\end{eqnarray}

Since the convergence of the finite element approximation to the
eigenvalue problem depends on the regularity of the original
eigenvalue problem, here we assume that the regularity of the
eigenfunction $u\in H^{1+\gamma}(\Omega)$  with  $0<\gamma\leq 1$
which is decided by the largest inner angle of the boundary
$\partial\Omega$.

For the eigenvalue problem, we have the following  basic expansion
from \cite{ArmentanoDuran}, and has been extensively used in
\cite{LinXie,Yang,YangZhangLin}.
\begin{lemma}\label{Eigenvalue_Expansion_Lemma}
 Suppose $(\lambda, u)$ is the eigenpair
of the original problem (\ref{laplaceproblem}), $(\lambda_h, u_h)\in
\mathcal{R}\times V_h$ is the eigenpair of the discrete problem
(\ref{dislaplaceproblem}), we have the following expansion
\begin{eqnarray}\label{eigenvalue_Expansion}
&&\hspace{-2cm}
\lambda-\lambda_h=\|u-u_h\|_{a,h}^2-\lambda_h\|v_h-u_h\|_b^2\nonumber\\
&&\ +\lambda_h(\|v_h\|_b^2-\|u\|_b^2)+2a_h(u-v_h,u_h),\quad \forall
v_h\in V_h.\label{expansion}
\end{eqnarray}
\end{lemma}

\section{Lower bound with nonconform finite element methods}
In this paper, we are concerned with two types of nonconforming
finite elements: Enriched Crouzeix-Raviart (ECR) (\cite{HuHuangLin,
LinXie}) and Extension $Q_1^{\rm rot}$ ($EQ_1^{\rm rot}$)
(\cite{Ltz04}) for triangle and rectangle partitions, respectively.
\begin{itemize}
\item ECR element is defined on the triangle partition and
\begin{align}
V_h^{NC}:= &\Big\{v\in L^2({\it  \Omega}):v|_{K}\in
{\rm span}\{1,x,y,x^2+y^2\},\int_{\ell} v|_{K_1}{\rm d}s=\int_{\ell} v|_{K_2}{\rm d}s,\nonumber\\
&\hspace{1cm}\text{ when }K_1\cap K_2=\ell\in\mathcal{E}_h,\
\text{and }\int_{\ell}v|_{K}{\rm d}s=0,  \text{ if } K\cap\partial
{\it \Omega}=\ell\Big\},
\end{align}
where $K,\ K_1,\ K_2\in \mathcal{T}_h$.
\item
On the rectangle partition, $EQ_1^{\rm rot}$ is employed
defined by
\begin{align}
V_h^{NC}:= &\Big\{v\in L^2(\Omega):v|_{K}\in
{\rm span}\{1,x,y,x^2,y^2\},\int_{\ell} v|_{K_1}ds=\int_{\ell} v|_{K_2}ds,\nonumber\\
&\hspace{1cm}\text{if }K_1\cap K_2=\ell, \text{ and
}\int_{\ell}v|_{K}ds=0,  \text{ if}\  K\cap\partial
\Omega=\ell\Big\},
\end{align}
where $K,\ K_1,\ K_2\in \mathcal{T}_h$.
\end{itemize}
From \cite{HuHuangLin} and \cite{LinXie}, the following basic error
estimates for the two nonconforming finite elements hold
\begin{eqnarray}
|\lambda-\lambda_h| &\leq& Ch^{2\gamma}\|{u}\|_{1+\gamma}^2,\label{Error_Eigenvalue}\\
\|u-u_{h}\|_{a,h}&\leq& Ch^{\gamma}\|u\|_{1+\gamma} \ ,\label{Error_Eigenfunction_1}\\
\|u-u_h\|_b &\leq& Ch^{\gamma}\|u-u_h\|_{a,h}\leq
Ch^{2\gamma}\|u\|_{1+\gamma}.\label{Error_Eigenfunction_2}
\end{eqnarray}

The interpolation operator corresponding to ECR and $EQ_1^{\rm rot}$
can be defined in the same way:
\begin{eqnarray}
\int_{\ell}(u-{\it \Pi}_h u)ds&=&0,\ \ \ \ \forall \ell\in\mathcal{E}_h,\label{condition_1}\\
\int_K(u-{\it \Pi}_h u){\rm d}K&=&0,\ \ \ \ \forall
K\in\mathcal{T}_h,\label{condition_2}
\end{eqnarray}
\begin{lemma}\label{ECR_Q1rot_Error_Lemma}
(See \cite{LinXie}) For the ECR and $EQ_1^{\rm rot}$ elements, the
interpolation operator satisfies:
$$\|u-{\it \Pi}_h u\|_b+h\|u-{\it \Pi}_h u\|_{a,h}\leq Ch^{1+\gamma}\|u\|_{1+\gamma},$$
for any $u\in H^{1+\gamma}(\Omega)$.
\end{lemma}

\begin{lemma}\label{ECR_EQ1rot_relation_Lemma}
With the interpolation defined above, for any $u\in V$, we have the
following results
\begin{eqnarray}
&&a_h(u-\it \Pi_h u,\it \Pi_h u)=0,\label{ainter}\\
&&\|\it \Pi_h u\|_{a,h}\leq C\|u\|_{a},\label{ainter1}\\
&&\|u-\it \Pi_h u\|_b\leq Ch\|u-\it \Pi_h u\|_{a,h},\label{ainter2}\\
&&\|u-\it \Pi_h u\|_b\leq C h \|u-u_h\|_{a,h}.\label{interpolation}
\end{eqnarray}
\end{lemma}
\begin{proof}
Here we only give the proof for the ECR element and the one for
$EQ_1^{\rm rot}$ is almost the same.

Integrating by parts, we have
 $$a_h(u-\it \Pi_h u,\it \Pi_h u)=\sum_{K\in\mathcal{T}_h}\big[\int_K(u-\it \Pi_h u)\triangle \Pi_h
 udxdy +\int_{\partial K}(u-\it \Pi_h u)(\nabla \it \Pi_h u)\cdot \textbf{n}ds\big].$$
Since $\it \Pi_h u\in V^{NC}_h$
$$\triangle \it \Pi_h u=const.$$
From the definition of the face interpolation, the following
equality holds
$$\int_K(u-\it \Pi_h u)\triangle \Pi_h udxdy=0.$$
For any $\ell\in \partial K$ can denoted by the linear function
$y=kx+m$, then the normal direction corresponding to this edge is
$\mathbf n=\frac{1}{\sqrt{1+k^2}}(\pm{k,-1})$. Thus suppose $\it
\Pi_h u=a+bx+cy+d(x^2+y^2)$, we have
\begin{eqnarray*}
(\nabla\it \Pi_h u)\cdot \mathbf
n&=&\pm(b+2dx,c+2dy)\frac{1}{\sqrt{1+k^2}}(k,-1)^{T}\nonumber\\
&=&\pm\frac{1}{\sqrt{1+k^2}}(bk-c-2dm)=const.
\end{eqnarray*} From the
definition of the  interpolation, we have
$$\int_{\ell}(u-\it \Pi_h u)(\nabla \it \Pi_h u)\cdot \mathbf nds=0.$$ 
So
$$a_h(u-\it \Pi_h u,\it\Pi_h u)=0.$$

From (\ref{ainter}), we can obtain $$a_h(\it \Pi_h u,\Pi_h
u)=a_h(u,\it \Pi_h u)\leq\|\it \Pi_h u\|_{a,h}\|u\|_{a,h}.$$
Canceling the term $\|\it \Pi_h u\|_{a,h}$ leads to (\ref{ainter1}).

Defining $\Pi_0$ be the piecewise constant interpolation, we have
\begin{eqnarray*}
\|u-\it \Pi_h u\|^2_b&=&b(u-\it \Pi_h u,u-\it \Pi_h u-\Pi_0(u-\it \Pi_h u))\\
&\leq &\|u-\it \Pi_h u\|_b\|u-\it \Pi_h u-\Pi_0(u-\it \Pi_h u)\|_b\\
&\leq&C\|u-\it \Pi_h u\|_bh\|u-\it \Pi_h u\|_{a,h}.
\end{eqnarray*}
It means we arrive the result  (\ref{ainter2}). 

For the last inequality, 
from (\ref{ainter1}) and (\ref{ainter2}), we have following equality
\begin{eqnarray*}
&&\|u-\it \Pi_h u\|_b\leq Ch \|u-\it \Pi_h u\|_{a,h} \leq C
h(\|u-u_h\|_{a,h}+\|u_h-\it \Pi_h
u\|_{a,h})\\
&&=Ch\big(\|u-u_h\|_{a,h} +\|\Pi_h u_h-\it \Pi_h u\|_{a,h}\big)\leq
Ch\|u-u_h\|_{a,h}.
\end{eqnarray*}
Thus, we get (\ref{interpolation}).
\end{proof}

\begin{lemma}(\cite[Section 3]{LinXieXu})\label{Lower_bound_Convergence_Order_Lemma}
If we solve the eigenvalue problem (\ref{weaklaplaceproblem}) by $Q_1^{\rm rot}$, ECR, bilinear or linear elements,
the following lower bound for the convergence rate holds
\begin{eqnarray}\label{Lower_bound_Convergence_Order}
\|u-u_h\|_{a,h}&\geq & Ch.
\end{eqnarray}
\end{lemma}

\begin{theorem}\label{Lower_Bound_Theorem}
Let $\lambda_j$ and $\lambda_{j,h}$ be the $j$-th exact eigenvalue
and its corresponding numerical approximation by ECR or $EQ_1^{\rm
rot}$ element. Assume $ u_j\in H^{1+\gamma}(\Omega)$ with $0<\gamma\leq
1$.  When $h$ is small enough, we have
\begin{eqnarray}\label{ECR}
0\leq \lambda_j-\lambda_{j,h}\leq Ch^{2\gamma}\|u\|_{1+\gamma}^2.
\end{eqnarray}
\end{theorem}
\begin{proof}
The result for ECR and $EQ_1^{\rm rot}$ elements can be proved in
the uniform way. We choose $v_{j,h}={\it\Pi_h} u_j$ in Lemma
\ref{Eigenvalue_Expansion_Lemma}. For the second term in
(\ref{eigenvalue_Expansion}), from Lemma
\ref{ECR_EQ1rot_relation_Lemma}, we have
$$\|u_{j,h}-\it \Pi_h u_j\|^2_b\leq C h^{
2} \|u_j-u_{j,h}\|^2_{a,h}.$$
 For the third term in (\ref{eigenvalue_Expansion}), from (\ref{interpolation}), we have
\begin{eqnarray*}
\|v_{j,h}\|_b^2-\|u_{j}\|_b^2&=&({\it\Pi_h} u_j-u_{j},{\it\Pi_h}
u_j+u_{j})\nonumber\\
&=&\big({\it\Pi_h} u_j-u_{j},({\it\Pi_h} u_{j,h}+u_{j})
-{\it\Pi_0}({\it\Pi_h} u_{j,h}+u_{j})\big)\\
&\leq& Ch \|{\it\Pi_h} u_j-u_{j}\|_b\leq
Ch^{2}\|u_j-u_{j,h}\|_{a,h},
\end{eqnarray*}
where $\Pi_0$ denotes the piecewise constant interpolation. Together
with (\ref{ainter}) and Lemma \ref{Lower_bound_Convergence_Order_Lemma}, the first term will be the dominant term in
(\ref{expansion}) and then (\ref{ECR}) can be derived.
\end{proof}

\section{Upper bound with linear conforming element}

\subsection{Upper bound by direct use of linear conforming element}

Due to the minimum-maximal principle (\ref{min_Max_principle}), by
using conforming finite element, we can get the upper bound
naturally. Here we use the lowest order conforming element for both
triangle and rectangle partitions, respectively
\begin{itemize}
\item On the triangle partition, we use the linear element defined  as below:
\begin{eqnarray}
V_h^{C}:= &\Big\{v\in C^0({\it  \Omega}):v|_{K}\in {\rm
span}\{1,x,y\}\Big\}\cap H^1_0(\Omega).
\end{eqnarray}
\item
On the rectangle partition, the bilinear element is employed here:
\begin{eqnarray}
V_h^{C}:= &\Big\{v\in C^0({\it  \Omega}):v|_{K}\in {\rm
span}\{1,x,y,xy\}\Big\}\cap H^1_0(\Omega).
\end{eqnarray}
\end{itemize}

In the conforming finite element space $V_h^C$, we can define the
corresponding discrete eigenvalue problem:

Find ($\overline{\lambda}_h,\overline{u}_h$)$\in \mathcal{R}\times
V_h^C$ such that $b(\overline{u}_h,\overline{u}_h)=1$ and
\begin{eqnarray}\label{LConform}
a(\overline{u}_h,\overline{v}_h)=\overline{\lambda}_h(\overline{u}_h,\overline{v}_h),
\quad \forall \overline{v}_h\in V_h^{C}.
\end{eqnarray}
From \cite{BA91}, the basic estimates about approximation errors
exist
\begin{eqnarray}
0\leq \overline{\lambda}_h-\lambda&\leq& Ch^{2\gamma}\|{u}\|_{1+\gamma}^2,\label{CError_Eigenvalue}\\
\|u-\overline{u}_{h}\|_{a,h}&\leq& Ch^{\gamma}\|u\|_{1+\gamma} \ ,\label{CError_Eigenfunction_1}\\
\|u-\overline{u}_h\|_b &\leq&
Ch^{2\gamma}\|u\|_{1+\gamma}.\label{CError_Eigenfunction_2}
\end{eqnarray}

So by (\ref{CError_Eigenvalue}), (\ref{Error_Eigenvalue}) and
Theorem \ref{Lower_Bound_Theorem}, we have the following estimate
between conforming and nonconforming eigenvalue approximations:
\begin{eqnarray}
0\leq\lambda-\lambda_h &\leq&
Ch^{2\gamma}\|{u}\|_{1+\gamma}^2,\label{Lower_error}\\
0\leq\overline{\lambda}_h-\lambda &\leq&
Ch^{2\gamma}\|{u}\|_{1+\gamma}^2,\label{Upper_Error}\\
0\leq\overline{\lambda}_h-\lambda_h&\leq&
Ch^{2\gamma}\|{u}\|_{1+\gamma}^2.\label{Lower_Upper_Error}
\end{eqnarray}

\subsection{Upper bound by postprocess}
The results (\ref{Lower_error})-(\ref{Lower_Upper_Error}) are very
interesting and useful. Especially, (\ref{Lower_Upper_Error}) can
give a guaranteed error for the current numerical approximations.
Unfortunately, in order to get (\ref{Lower_Upper_Error}), we need to
solve the eigenvalue problem twice, one with nonconforming and the
other with conforming elements, which is always difficult since
solving eigenvalue problem need much more computation than solving
the corresponding source problems. So, in this subsection we give a
postprocessing method to obtain the upper bound eigenvalue
approximation without solving eigenvalue problem but only need to
solve a source problem.

After obtaining the eigenpair approximation $(\lambda_h,u_h)\in
\mathcal{R}\times V_h^{NC}$ by the nonconforming finite element, we
put them as the right hand side of an auxiliary source Laplace
problem. Then we utilize conforming finite element method to solve
 this  problem which is defined as follows:
\begin{equation}\label{CFlaplaceproblem}
 a(\widehat{u}_h,\widehat{v}_h)=\lambda_h b(u_h,\widehat{v}_h),\  \quad \forall \widehat{v}_h\in V_h^{C}
\end{equation}
After obtaining $\widehat{u}_h$, we calculate the following Rayleigh
quotient as an approximation of $\lambda$
\begin{equation}\label{Conform}
\widehat{\lambda}_h:=R(\widehat{u}_h).
\end{equation}
In order to analyze the error estimate for the eigenpair
approximation
$(\widehat{\lambda}_h,\widehat{u}_h)\in\mathcal{R}\times V_h^C$, we
define the projection operator associated with the space $V_h^C$ as
follows
\begin{eqnarray}\label{Projection_Operator}
a(\widehat{P}_hu,\widehat{v}_h)&=&a(u,\widehat{v}_h),\ \ \ \forall
\widehat{v}_h\in V_h^C.
\end{eqnarray}
For this projection operator, we have the following error estimate
\begin{eqnarray}\label{Projection_Operator_Error}
\|\widehat{P}_hu-u\|_b+h^{\gamma}\|\widehat{P}_hu-u\|_a&\leq&Ch^{2\gamma}\|u\|_{1+\gamma}.
\end{eqnarray}

 An important lemma from Babu\v{s}ka and Osborn (\cite{BA91, XuZhou}), tells that
\begin{lemma}(\cite[Lemma 4.1]{BA91})
For the self-adjoint problem (\ref{dislaplaceproblem}), suppose
($\lambda, u$) be the exact eigenpair. Then for any $w\in V$,
$\|w\|_b\neq 0$, the Rayleigh quotient $R(w)$ satisfy
\begin{equation}\label{Rayleighupper}
R(w)-\lambda=\frac{\|w-u\|_a^2}{\|w\|_b^2}-\lambda\frac{\|w-u\|_b^2}{\|w\|_b^2}.
\end{equation}
\end{lemma}

\begin{theorem}\label{Upper_Bound_Theorem_1}
For the eigenpair approximation
$(\widehat{\lambda}_h,\widehat{u}_h)\in\mathcal{R}\times
\widehat{V}_h^C$, the following error estimates hold
\begin{eqnarray}
\|u-\widehat{u}_h\|_a&\leq&Ch^{\gamma}\|u\|_{1+\gamma},\label{Error_conforming_eigenfunction}\\
|\widehat{\lambda}_h-\lambda|&\leq&Ch^{2\gamma}.\label{Error_conforming_eigenvalue}
\end{eqnarray}
Assume $u\in H^{1+\gamma}(\Omega)$ with $0 <\gamma\leq 1$.
Then when $h$ is small enough, we have the following upper bound for
$\widehat{\lambda}_h$ defined by (\ref{upper}),
\begin{eqnarray}\label{upper_eigenvalue}
\widehat{\lambda}_h& \geq &\lambda.
\end{eqnarray}
\end{theorem}
\begin{proof}
First, we have the following estimate
\begin{eqnarray*}
\|\widehat{u}_h-\widehat{P}_hu\|_a^2
&=&a(\widehat{u}_h-\widehat{P}_hu,\widehat{u}_h-\widehat{P}_hu)=b(\lambda_hu_h-\lambda
u,\widehat{u}_h-\widehat{P}_hu)\\
&\leq&\|\lambda_hu_h-\lambda u\|_{0,h}\|\widehat{u}_h-\widehat{P}_hu\|_a\nonumber\\
&\leq&\big(|\lambda_h|\|u-u_h\|_b+|\lambda_h-\lambda|\|u\|_b\big)\|\widehat{u}_h-\widehat{P}_hu\|_a\\
&\leq&Ch^{2\gamma}\|\widehat{u}_h-\widehat{P}_hu\|_a\|u\|_{1+\gamma}.
\end{eqnarray*}
So the following estimate holds
\begin{eqnarray}\label{Super_u_h_P_h_u}
\|\widehat{u}_h-\widehat{P}_hu\|_a&\leq&Ch^{2\gamma}\|u\|_{1+\gamma}.
\end{eqnarray}
Then we have the following error estimates for $\widehat{u}_h$
\begin{eqnarray*}
\|\widehat{u}_h-u\|_a&\leq&\|\widehat{u}_h-\widehat{P}_hu\|_a+\|\widehat{P}_hu-u\|_a\nonumber\\
&\leq&Ch^{\gamma}\|u\|_{1+\gamma},\\
\|\widehat{u}_h-u\|_b&\leq&\|\widehat{u}_h-\widehat{P}_hu\|_a+\|\widehat{P}_hu-u\|_b\nonumber\\
&\leq&Ch^{2\gamma}\|u\|_{1+\gamma}.
\end{eqnarray*}
Since $V_h^{C}\subset V$, replacing $w$ with $\widehat{u}_h$ in
(\ref{Rayleighupper}), we have
\begin{equation}\label{upper}
\widehat{\lambda}_h-\lambda=\frac{\|\widehat{u}_h-u\|_a^2}{\|\widehat{u}_h\|_b^2}
-\lambda\frac{\|\widehat{u}_h-u\|_b^2}{\|\widehat{u}_h\|_b^2}.
\end{equation}
Applying the estimates of $\widehat{u}_h$ to (\ref{upper}), we can
obtain (\ref{Error_conforming_eigenvalue}). Furthermore from Lemma \ref{Lower_bound_Convergence_Order_Lemma},
 we have $\|u-\widehat{u}_h\|_{a}\geq Ch$. Thus we can see that the first term is dominate and the
desired result (\ref{upper_eigenvalue}) is derived.
\end{proof}

\begin{theorem}\label{Two_Side_Theorem_1}
Under the conditions in Theorem \ref{Lower_Bound_Theorem}  and
\ref{Upper_Bound_Theorem_1}, for the eigenvalue approximation
$\lambda_h$ and $\widehat{\lambda}_h$,  we have $0\leq
\widehat{\lambda}_h-\lambda_h\leq Ch^{2\gamma}$ and
$\max\{|\lambda-\lambda_h|, |\widehat{\lambda}_h-\lambda|\}\leq
\widehat{\lambda}_h-\lambda_h\leq Ch^{2\gamma}$.
\end{theorem}
\begin{proof}
Based on the results in Theorems \ref{Lower_Bound_Theorem} and
\ref{Upper_Bound_Theorem_1} and the error estimate
(\ref{Error_Eigenvalue}), we can easily prove this theorem.
\end{proof}

\section{Better upper bound approximation with finer finite element space}
In the last section, the postprocessing method is applied to get the
upper bound of the eigenvalue which has the same convergence order
as the lower bound. But as we know from \cite{XuZhou}, we can employ
the postprocessing method to obtain a new eigenpair approximation
with better accuracy than the obtained nonconforming approximation
$(\lambda_h,u_h)$. But because of the higher order convergence, the
analysis of upper bound in last section can not be used in this
case. Here, we propose a method to produce not only
higher order accuracy but also upper bound approximation of the
eigenvalue. This procedure contains solving some auxiliary source
problems and a very small eigenvalue problem.

The aim of this section is to obtain the approximation of the first
$m$ eigenvalues $\lambda_1\leq\lambda_2\leq\cdots\leq \lambda_m$.
Assume we have obtained the first eigenvalue approximations
$\lambda_{1,h}\leq \lambda_{2,h}\cdots\leq \lambda_{m,h}$ and the
corresponding eigenfunction approximations $u_{1,h}, u_{2,h},\cdots,
u_{m,h}$ by the nonconforming finite element method with the lower
bound property $\lambda_{j,h}\leq \lambda_j\ (j=1,2,\cdots,m)$.

Before introducing the postprocessing method, we need to define a
finer conforming finite element space $V_h^{HC}$ which can be
constructed by using higher order finite element or refining current
mesh $\mathcal{T}_h$ such that this space has the higher convergence
order (\cite{XuZhou})
\begin{eqnarray}\label{Higer_Order_Convergence}
\inf_{\widehat{v}_h\in
\widehat{V}_h^{HC}}\|u-\widehat{v}_h\|_a&\leq&Ch^{2\gamma},\ \ \
\forall u\in H^{1+\gamma}(\Omega).
\end{eqnarray}

Now let us introduce a postprocessing method to get the upper bound
approximations with the help of obtained approximations
$(\lambda_{1,h},u_{1,h}),\cdots,(\lambda_{m,h},u_{m,h})$.

\begin{algorithm}\label{Higher_Order_Upper_Algorithm}
Higher order postprocessing method:

\begin{enumerate}
\item For $j=1,2,\cdots, m$

Find $\widehat{u}_{j,h}\in V_h^{HC}$ such that
\begin{eqnarray}\label{auxilary_equations}
a(\widehat{u}_{j,h},\widehat{v}_h) &=&
\lambda_{j,h}b(u_{j,h},\widehat{v}_h),\ \ \ \ \forall
\widehat{v}_h\in V_h^{HC}.
\end{eqnarray}
\item Construct the finite dimensional space $\widehat{V}_h = {\rm span}\{\widehat{u}_{1,h},\cdots,\widehat{u}_{m,h}\}$
and solve the following eigenvalue problem in the space
$\widehat{V}_h$:

Find $(\widetilde{\lambda}_h,\widetilde{u}_h)\in \mathcal{R}\times
\widehat{V}_h$ such that
\begin{eqnarray}\label{Aux_Eigenvalue_Problem}
a(\widetilde{u}_h,\widetilde{v}_h)&=&\widetilde{\lambda}_hb(\widetilde{u}_h,\widetilde{v}_h),\
\ \ \ \forall \widetilde{v}\in \widehat{V}_h.
\end{eqnarray}
Finally, we obtain the new eigenpair approximations
$(\widetilde{\lambda}_{1,h},\widetilde{u}_{1,h}),\cdots,(\widetilde{\lambda}_{m,h},\widetilde{u}_{m,h})$.
\end{enumerate}
\end{algorithm}
Similarly to the last section, in order to analyze the error
estimate for the eigenfunction approximations
$\widehat{u}_{1,h},\cdots,\widehat{u}_{m,h}$, we define the
projection operator corresponding to the space $V_h^{HC}$ as follows
\begin{eqnarray}\label{Projection_Operator_high}
a(P_h^{HC}u,\widehat{v}_h)&=&a(u,\widehat{v}_h),\ \ \ \forall
\widehat{v}_h\in V_h^{HC}.
\end{eqnarray}
For this projection operator, we have the following error estimate
\begin{eqnarray}\label{Projection_Operator_Error_high}
\|P_h^{HC}u-u\|_a &\leq&
C\inf_{\widehat{v}_h\in\widehat{V}_h^{HC}}\|u-\widehat{v}_h\|_a \leq
Ch^{2\gamma}.
\end{eqnarray}

\begin{theorem}\label{Higher_Order_Upper}
For the eigenpair approximations
$(\widetilde{\lambda}_{1,h},\widetilde{u}_{1,h}),\cdots,
(\widetilde{\lambda}_{m,h},\widetilde{u}_{m,h})$ obtained by
Algorithm \ref{Higher_Order_Upper_Algorithm}, we have following
estimates
\begin{eqnarray}
0\leq \widetilde{\lambda}_{j,h}-\lambda_{j}&\leq& Ch^{4\gamma},\label{Higher_order_Upper_Eigenvalue}\\
\|\widetilde{u}_{j,h}-u\|_a &\leq & Ch^{2\gamma}.
\label{Higher_Order_Eigenfunction}
\end{eqnarray}
\end{theorem}

\begin{proof}
First, we prove the error estimate of the eigenfunction
approximations (\ref{Higher_Order_Eigenfunction}).  From
(\ref{weaklaplaceproblem}), (\ref{Projection_Operator_high}) and the
coercivity of $a(\cdot,\cdot)$, we have
\begin{eqnarray*}
\|\widehat{u}_{j,h}-P_h^{HC}u_j\|_a^2
&=&a(\widehat{u}_{j,h}-P_h^{HC}u_j,\widehat{u}_{j,h}-P_h^{HC}u_j)
=b(\lambda_{j,h}u_{j,h}-\lambda_j
u_j,\widehat{u}_{j,h}-P_h^{HC}u_j)\\
&\leq&\|\lambda_{j,h}u_{j,h}-\lambda_j u_j\|_{0,h}\|\widehat{u}_{j,h}-P_h^{HC}u_j\|_a\nonumber\\
&\leq&\big(|\lambda_{j,h}|\|u_j-u_{j,h}\|_b+|\lambda_{j,h}-\lambda_j|\|u_j\|_b\big)\|\widehat{u}_{j,h}
-P_h^{HC}u_j\|_a\\
&\leq&Ch^{2\gamma}\|\widehat{u}_{j,h}-P_h^{HC}u_j\|_a.
\end{eqnarray*}
Combined with (\ref{Projection_Operator_Error_high}), the following
estimate holds
\begin{eqnarray}
\|\widehat{u}_{j,h}-u_j\|_a &\leq&
\|\widehat{u}_{j,h}-P_h^{HC}u_j\|_a+\|P_h^{HC}u-u\|_a\leq
Ch^{2\gamma}.
\end{eqnarray}
Based on the theory in \cite{BA91} for the error estimate of the
eigenvalue problem by finite element method, we have the following
inequality
\begin{eqnarray}
\|\widetilde{u}_{j,h}-u_j\|_a&\leq&C\inf_{\widetilde{v}_h\in\widehat{V}_h}\|u_j-\widetilde{v}_h\|_a
\leq C\|u_j-\widehat{u}_{j,h}\|_a\leq Ch^{2\gamma}.
\end{eqnarray}
This is the desired result (\ref{Higher_Order_Eigenfunction}). Now,
we come to prove the error estimate property
(\ref{Higher_order_Upper_Eigenvalue}). With the help of
(\ref{Rayleighupper}), we have
\begin{eqnarray}
|\widetilde{\lambda}_{j,h}-\lambda_j|&\leq&C\|\widetilde{u}_{j,h}-u_j\|_a^2\leq
Ch^{4\gamma}.
\end{eqnarray}
Thanks to the minimum-maximum principle (\ref{min_Max_principle}),
we have
\begin{eqnarray}\label{min_max_high_order_1}
\lambda_j=\min\limits_{V_j\subset V,\dim V_j=j}\max\limits_{v\in
V_j}R(v),
\end{eqnarray}
and the discrete version in the space $\widehat{V}_h$
\begin{eqnarray}\label{min_max_high_order_2}
\widetilde{\lambda}_{j,h}=\min_{V_{j,h}\subset \widehat{V}_h,\dim
V_{j,h}=j}\max\limits_{v\in V_{j,h}}R(v).
\end{eqnarray}
Based on (\ref{min_max_high_order_1}) and
(\ref{min_max_high_order_2}), we can easily obtain  $\lambda_j\leq
\widetilde{\lambda}_{j,h}\ (j=1,2,\cdots,m)$ and complete the proof.
\end{proof}

\section{Numerical Results}
In this section, two numerical examples are presented to validate
our theoretical results stated in the above sections.

\subsection{Eigenvalue problem on the unit square domain}
In this subsection, we solve the eigenvalue problem
(\ref{laplaceproblem}) on the unit domain $\Omega=(0,1)\times
(0,1)$. The aim here is to find the approximations of the first $6$
eigenvalues $\lambda_1, \lambda_2,\cdots, \lambda_6$.

First, ECR element is applied to solve the eigenvalue problem and
then the linear finite element to do the postprocessing on the
series of meshes which are produced by Delaunay scheme. The
quadratic element is applied to implement Algorithm
\ref{Higher_Order_Upper_Algorithm}. Table \ref{ECR_Square} shows the
eigenvalue approximations of the first $6$ eigenvalues. And the
approximations by postprocessing method with
 linear element is presented in Table \ref{Linear_Square}. Table
 \ref{Quadratic_Square} shows the numerical results of the
 postprocessing Algorithm \ref{Higher_Order_Upper_Algorithm} with quadratic element.
 From Table \ref{ECR_Square}, we can find the numerical approximations of ECR element are
 lower bounds of the exact eigenvalues. Tables \ref{Linear_Square} and
 \ref{Quadratic_Square} show the upper bounds of the numerical
 approximations by the postprocessing method using linear and quadratic
 elements.

\begin{table}[htb!]
\centering \caption{ECR element for eigenvalue problem on unit
square}\label{ECR_Square}
\begin{tabular}{|c|c|c|c|c|c|c|}\hline
$h$&$\lambda_{1,h}$&$\lambda_{2,h}$&$\lambda_{3,h}$&$\lambda_{4,h}$ &$\lambda_{5,h}$ &$\lambda_{6,h}$ \\
\hline $0.2$&     19.282902 & 46.704835 &  46.785588 & 72.047988  &   88.938276 & 89.192096   \\
\hline $0.1$&     19.609902 &  48.619503 & 48.628462 & 77.120773  &   95.649531 & 95.991202 \\
\hline $0.05$&    19.711840 &  49.180465 & 49.181471 & 78.517381  &   98.019716 & 98.037715 \\
\hline $0.025$&   19.732395 &  49.306198 & 49.306523 & 78.849396  &   98.532500 & 98.532943 \\
\hline $0.0125$&  19.737526 &  49.337705 & 49.337756 & 78.930339  &   98.655173 & 98.655315 \\
\hline Trend & $\nearrow $ & $\nearrow $   & $\nearrow$ & $\nearrow$ &  $\nearrow$& $\nearrow$\\
\hline
\end{tabular}
\end{table}

\begin{table}[htb!]
\centering \caption{Linear element for postprocessing method on unit
square}\label{Linear_Square}
\begin{tabular}{|c|c|c|c|c|c|c|}\hline
$h$&$\lambda_{1,h}$&$\lambda_{2,h}$&$\lambda_{3,h}$&$\lambda_{4,h}$ &$\lambda_{5,h}$ &$\lambda_{6,h}$ \\
\hline $0.2$&      20.458912 &  53.558232 &  53.776851&  90.030803 &    117.15806 &  116.15773   \\
\hline $0.1$&      19.940425 &  50.568044 &  50.561690&  82.204934 &    103.84878 &  103.61419 \\
\hline $0.05$&     19.784700 &  49.634441 &  49.633984&  79.701583 &    99.844154 &  99.842216 \\
\hline $0.025$&    19.750665 &  49.419822 &  49.419553&  79.140201 &    98.978970 &  98.978102 \\
\hline $0.0125$&   19.742054 &  49.365577 &  49.365546&  79.002394 &    98.765631 &  98.765878 \\
\hline Trend & $\searrow $ & $\searrow $   & $\searrow$ & $\searrow$ &  $\searrow$& $\searrow$\\
\hline
\end{tabular}
\end{table}
\begin{table}[htb!]
\centering \caption{Quadratic element for postprocessing method on
unit square}\label{Quadratic_Square}
\begin{tabular}{|c|c|c|c|c|c|c|}\hline
$h$&$\lambda_{1,h}$&$\lambda_{2,h}$&$\lambda_{3,h}$&$\lambda_{4,h}$ &$\lambda_{5,h}$ &$\lambda_{6,h}$ \\
\hline $0.2$&       19.744698 &  49.427336 &   49.434146&   79.315293 &    99.263888 &   99.326679   \\
\hline $0.1$&       19.739654 &  49.354899 &   49.354991&   78.982884 &    98.747751 &   98.748129 \\
\hline $0.05$&      19.739231 &  49.348360 &   49.348371&   78.958190 &    98.698705 &   98.698827 \\
\hline $0.025$&     19.739210 &  49.348043 &   49.348043&   78.956922 &    98.696216 &   98.696219 \\
\hline $0.0125$&    19.739209 &  49.348023 &   49.348023&   78.956841 &    98.696054 &   98.696054 \\
\hline Trend & $\searrow $ & $\searrow $   & $\searrow$ & $\searrow$ &  $\searrow$& $\searrow$\\
\hline
\end{tabular}
\end{table}

Then, $EQ_1^{\rm rot}$ element is applied to solve the eigenvalue
problem and then the bilinear finite element to do the
postprocessing on the series of uniform rectangle meshes.
Biquadratic element is employed to implement Algorithm
\ref{Higher_Order_Upper_Algorithm}.  Table \ref{EQ1Rot_Square} shows
the eigenvalue approximations of the first $6$ eigenvalues and the
approximations by postprocessing method with
 bilinear element is presented in Table \ref{Bilinear_Square}. Table
 \ref{Biquadratic_Square} shows the numerical results of the
 postprocessing Algorithm \ref{Higher_Order_Upper_Algorithm} with biquadratic element.
 From Table \ref{EQ1Rot_Square}, we can find the numerical approximations of $EQ_1^{\rm rot}$ element are
 lower bounds of the exact eigenvalues. Tables \ref{Bilinear_Square} and
 \ref{Biquadratic_Square} show the upper bounds of the numerical
 approximations by the postprocessing method using bilinear and biquadratic
 elements.
\begin{table}[htb!]
\centering \caption{$EQ_1^{\rm rot}$ element for eigenvalue problem on unit
square}\label{EQ1Rot_Square}
\begin{tabular}{|c|c|c|c|c|c|c|}\hline
$h$&$\lambda_{1,h}$&$\lambda_{2,h}$&$\lambda_{3,h}$&$\lambda_{4,h}$ &$\lambda_{5,h}$ &$\lambda_{6,h}$ \\
\hline $1/8\times 1/8$&      19.530807 &  47.509983 & 47.523306 &  75.754849  &    89.869065 & 89.943840   \\
\hline $1/16\times 1/16$&      19.686551 &  48.878512 & 48.879400 &  78.123216  &    96.394505 & 96.399928 \\
\hline $1/32\times 1/32$&     19.726009 &  49.229994 & 49.230051 &  78.746202  &    98.114256 & 98.114609 \\
\hline $1/64\times 1/64$&    19.735907 &  49.318474 & 49.318478 &  78.904035  &    98.550189 & 98.550211 \\
\hline $1/128\times 1/128$&   19.738383 &  49.340632 & 49.340633 &  78.943626  &    98.659555 & 98.659556 \\
\hline Trend & $\nearrow $ & $\nearrow $   & $\nearrow$ & $\nearrow$ &  $\nearrow$& $\nearrow$\\
\hline
\end{tabular}
\end{table}

\begin{table}[htb!]
\centering \caption{Bilinear element for postprocessing method on unit
square}\label{Bilinear_Square}
\begin{tabular}{|c|c|c|c|c|c|c|}\hline
$h$&$\lambda_{1,h}$&$\lambda_{2,h}$&$\lambda_{3,h}$&$\lambda_{4,h}$ &$\lambda_{5,h}$ &$\lambda_{6,h}$ \\
\hline$1/8\times 1/8$&       20.506335 &   52.821611 &   54.413322&  90.985945 &     114.01626 &   114.01577   \\
\hline$1/16\times 1/16$&     19.929848 &   50.211622 &   50.588305&  82.000375 &     102.46652 &   102.46652 \\
\hline$1/32\times 1/32$&     19.786796 &   49.563576 &   49.656364&  79.717934 &     99.633341 &   99.633341 \\
\hline$1/64\times 1/64$&     19.751101 &   49.401888 &   49.424998&  79.147095 &     98.930015 &   98.930015 \\
\hline$1/128\times 1/128$&   19.742182 &   49.361487 &   49.367259&  79.004399 &     98.754514 &   98.754514 \\
\hline Trend & $\searrow $ & $\searrow $   & $\searrow$ & $\searrow$ &  $\searrow$& $\searrow$\\
\hline
\end{tabular}
\end{table}
\begin{table}[htb!]
\centering \caption{Biquadratic element for postprocessing method on
unit square}\label{Biquadratic_Square}
\begin{tabular}{|c|c|c|c|c|c|c|}\hline
$h$&$\lambda_{1,h}$&$\lambda_{2,h}$&$\lambda_{3,h}$&$\lambda_{4,h}$ &$\lambda_{5,h}$ &$\lambda_{6,h}$ \\
\hline$1/8\times 1/8$&       19.743647&   49.388394&    49.422024&    79.218966&     99.081741 &    99.083312   \\
\hline$1/16\times 1/16$&     19.739492&   49.350652&    49.352826&    78.974583&     98.721403 &    98.721409 \\
\hline$1/32\times 1/32$&     19.739227&   49.348188&    49.348325&    78.957968&     98.697654 &    98.697654 \\
\hline$1/64\times 1/64$&     19.739210&   49.348032&    49.348041&    78.956906&     98.696145 &    98.696145 \\
\hline$1/128\times 1/128$&   19.739209&   49.348023&    49.348023&    78.956840&     98.696050 &    98.696050 \\
\hline Trend & $\searrow $ & $\searrow $   & $\searrow$ & $\searrow$ &  $\searrow$& $\searrow$\\
\hline
\end{tabular}
\end{table}
Figure \ref{Error_UnitSquare} shows the errors of the eigenvalue
approximations by ECR and $EQ_1^{\rm rot}$ elements, postprocessing
methods with lowest order (linear and bilinear) and higher order
(quadratic and biquadratic) elements on the unit square. Since we
know the exact eigenvalues on the unit square, we can give the exact
errors of $\lambda_{j,h}$, $\widehat{\lambda}_{j,h}$ and
$\widetilde{\lambda}_{j,h}\ (j=1,2,3,4,5,6)$. Since the
eigenfunctions are smooth, the postprocessing with higher order
element can improve the convergence order. From Figure
\ref{Error_UnitSquare}, we can find the eigenvalue approximations
have the reasonable convergence order.
\begin{figure}[ht]
\centering
\includegraphics[width=7cm,height=7cm]{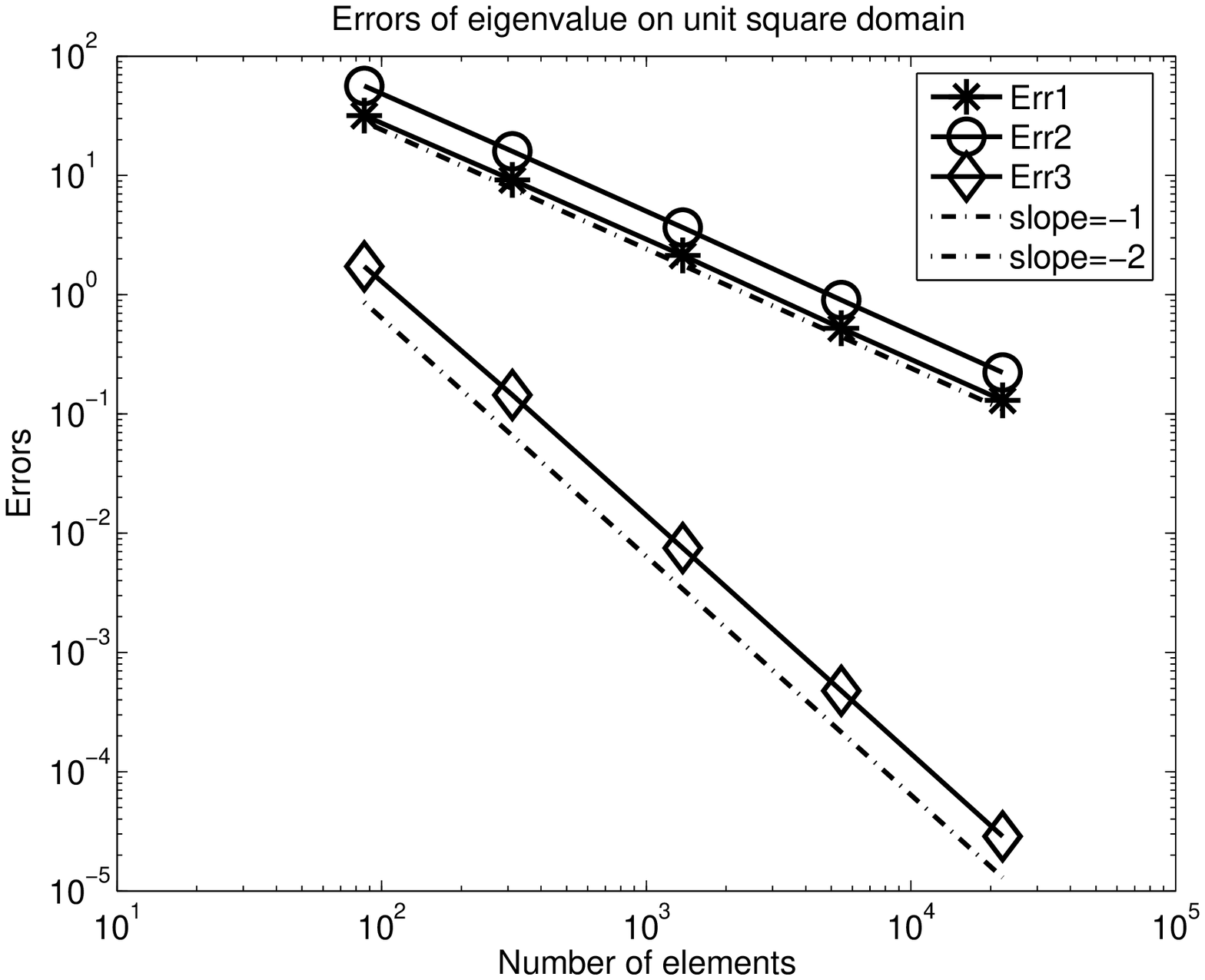}
\includegraphics[width=7cm,height=7cm]{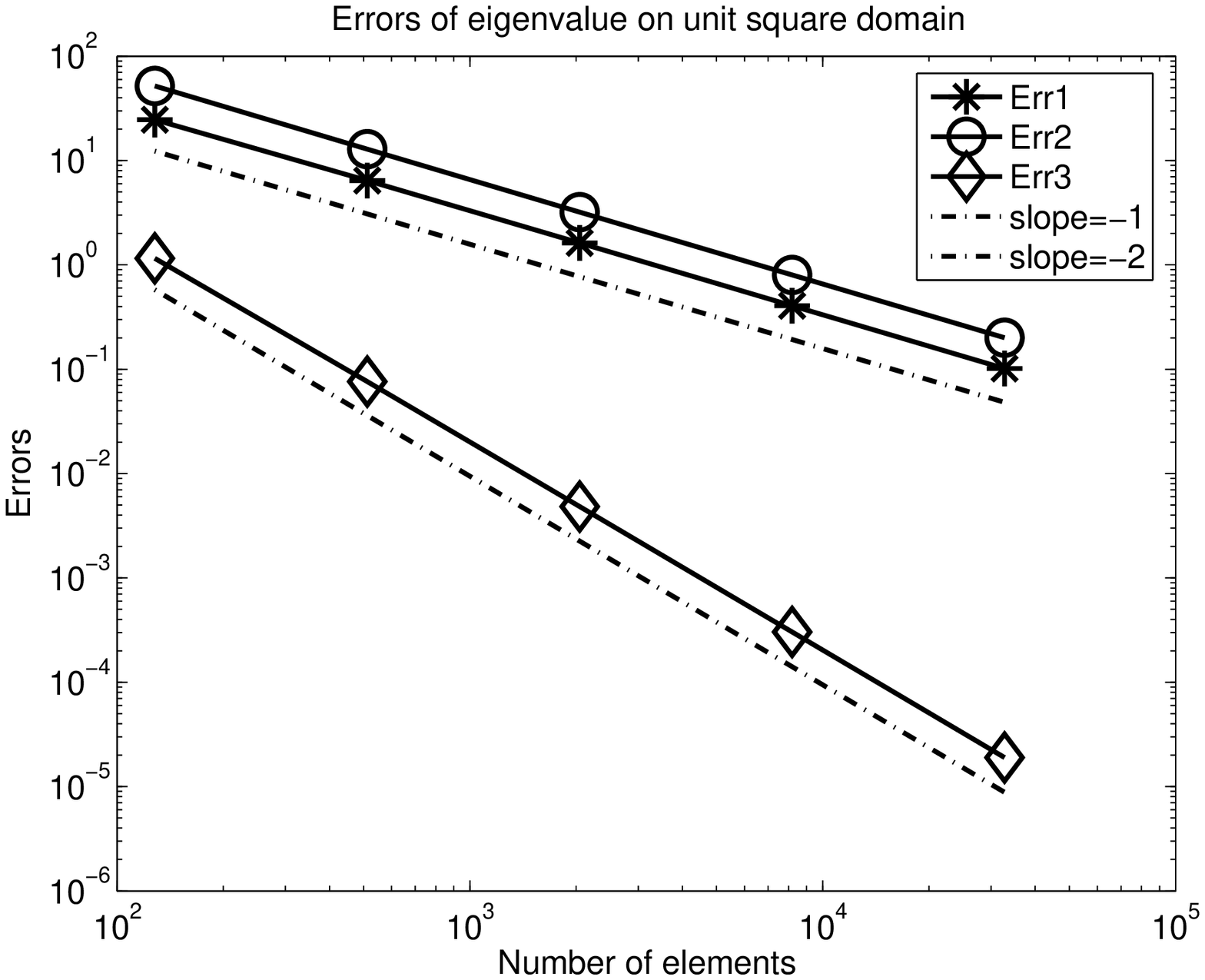}
\caption{The errors for the eigenvalue approximations on unit square
by ECR (left) and $EQ_1^{\rm rot}$ (right), where ${\rm
Err1}=\sum_{j=1}^6(\lambda_j-\lambda_{j,h})$, ${\rm
Err2}=\sum_{j=1}^6(\widehat{\lambda}_{j,h}-\lambda_j)$ and ${\rm
Err3}=\sum_{j=1}^6(\widetilde{\lambda}_{j,h}-\lambda_j)$}
\label{Error_UnitSquare}
\end{figure}

\subsection{Eigenvalue problem on the $L$ shape domain}
In this subsection, we solve the eigenvalue problem
(\ref{laplaceproblem}) on the $L$ shape domain  $\Omega=(-1,
1)\times (-1, 1)\setminus (-1,0)\times(-1,0)$. The aim here is also
to find the approximations of the first $6$ eigenvalues $\lambda_1,
\lambda_2,\cdots, \lambda_6$.

First, ECR element is applied to solve the eigenvalue problem and
then the linear finite element to do the postprocessing on the
series of meshes which are produced by Delaunay scheme. The
quadratic element is applied to implement Algorithm
\ref{Higher_Order_Upper_Algorithm}. Table \ref{ECR_LShape} shows the
eigenvalue approximations of the first $6$ eigenvalues and the
approximations by postprocessing method with
 linear element is presented in Table \ref{Linear_LShape}. Table
 \ref{Quadratic_LShape} shows the numerical results of the
 postprocessing Algorithm \ref{Higher_Order_Upper_Algorithm} with quadratic element.
 From Table \ref{ECR_LShape}, we can find the numerical approximations of ECR element are
 lower bounds of the exact eigenvalues. Tables \ref{Linear_LShape} and
 \ref{Quadratic_LShape} show the upper bounds of the numerical
 approximations by the postprocessing method using linear and quadratic
 elements.

\begin{table}[htb!]
\centering \caption{ECR element for eigenvalue problem on $L$ shape
domain}\label{ECR_LShape}
\begin{tabular}{|c|c|c|c|c|c|c|}\hline
$h$&$\lambda_{1,h}$&$\lambda_{2,h}$&$\lambda_{3,h}$&$\lambda_{4,h}$ &$\lambda_{5,h}$ &$\lambda_{6,h}$ \\
\hline $0.2$&      8.9126839 &  14.379736 &  18.200634 &  26.360803  &    27.160361 &  34.333785   \\
\hline $0.1$&      9.3655553 &  14.954062 &  19.323660 &  28.615430  &    30.449112 &  39.363462 \\
\hline $0.05$&     9.5439611 &  15.133230 &  19.625778 &  29.266600  &    31.451422 &  40.860929 \\
\hline $0.025$&    9.6066551 &  15.181407 &  19.711040 &  29.459256  &    31.775515 &  41.305926 \\
\hline $0.0125$&   9.6277308 &  15.193312 &  19.732422 &  29.506604  &    31.869752 &  41.426233 \\
\hline Trend & $\nearrow $ & $\nearrow $   & $\nearrow$ & $\nearrow$ &  $\nearrow$& $\nearrow$\\
\hline
\end{tabular}
\end{table}

\begin{table}[htb!]
\centering \caption{Linear element for postprocessing method on $L$
shape domain}\label{Linear_LShape}
\begin{tabular}{|c|c|c|c|c|c|c|}\hline
$h$&$\lambda_{1,h}$&$\lambda_{2,h}$&$\lambda_{3,h}$&$\lambda_{4,h}$ &$\lambda_{5,h}$ &$\lambda_{6,h}$ \\
\hline $0.2$&       10.578106 &   16.642707 &   22.460586&   36.034300 &    39.441562 &   52.785280   \\
\hline $0.1$&       9.9724152 &   15.610430 &   20.469954&   31.132236 &    34.109328 &   44.593560 \\
\hline $0.05$&      9.7438532 &   15.309456 &   19.934653&   29.958076 &    32.522745 &   42.389559 \\
\hline $0.025$&     9.6728756 &   15.224811 &   19.786815&   29.627051 &    32.086619 &   41.719036 \\
\hline $0.0125$&    9.6521963 &   15.203927 &   19.750613&   29.546869 &    31.965675 &   41.541732 \\
\hline Trend & $\searrow $ & $\searrow $   & $\searrow$ & $\searrow$ &  $\searrow$& $\searrow$\\
\hline
\end{tabular}
\end{table}
\begin{table}[htb!]
\centering \caption{Quadratic element for postprocessing method on
$L$ shape domain}\label{Quadratic_LShape}
\begin{tabular}{|c|c|c|c|c|c|c|}\hline
$h$&$\lambda_{1,h}$&$\lambda_{2,h}$&$\lambda_{3,h}$&$\lambda_{4,h}$ &$\lambda_{5,h}$ &$\lambda_{6,h}$ \\
\hline $0.2$&        9.7067976 &   15.233894 &    19.813091&   29.809049 &     32.402601 &   42.472832  \\
\hline $0.1$&        9.6691745 &   15.201608 &    19.744920&   29.541711 &     32.011248 &   41.586610\\
\hline $0.05$&       9.6496097 &   15.197646 &    19.739617&   29.522960 &     31.939132 &   41.497609\\
\hline $0.025$&      9.6432648 &   15.197293 &    19.739232&   29.521572 &     31.921467 &   41.481392\\
\hline $0.0125$&     9.6414840 &   15.197258 &    19.739210&   29.521488 &     31.916952 &   41.477779\\
\hline Trend & $\searrow $ & $\searrow $   & $\searrow$ & $\searrow$ &  $\searrow$& $\searrow$\\
\hline
\end{tabular}
\end{table}

Then, $EQ_1^{\rm rot}$ element is applied to solve the eigenvalue
problem and then the bilinear finite element to do the
postprocessing on the series of uniform rectangle meshes.
Biquadratic element is employed to implement Algorithm
\ref{Higher_Order_Upper_Algorithm}.  Table \ref{EQ1Rot_LShape} shows
the eigenvalue approximations of the first $6$ eigenvalues and the
approximations by postprocessing method with
 bilinear element is presented in Table \ref{Bilinear_LShape}. Table
 \ref{Biquadratic_LShape} shows the numerical results of the
 postprocessing Algorithm \ref{Higher_Order_Upper_Algorithm} with biquadratic element.
 From Table \ref{EQ1Rot_LShape}, we can find the numerical approximations of $EQ_1^{\rm rot}$ element are
 lower bounds of the exact eigenvalues. Tables \ref{Bilinear_LShape} and
 \ref{Biquadratic_LShape} show the upper bounds of the numerical
 approximations by the postprocessing method using bilinear and biquadratic
 elements.

\begin{table}[htb!]
\centering \caption{$EQ_1^{\rm rot}$ element for eigenvalue problem
on $L$ shape domain}\label{EQ1Rot_LShape}
\begin{tabular}{|c|c|c|c|c|c|c|}\hline
$h$&$\lambda_{1,h}$&$\lambda_{2,h}$&$\lambda_{3,h}$&$\lambda_{4,h}$ &$\lambda_{5,h}$ &$\lambda_{6,h}$ \\
\hline $1/4\times 1/4$&       9.2784846 &   15.049120 &  19.477978&   28.869068  &     30.381898 &  39.579205   \\
\hline $1/8\times 1/8$&       9.5063501 &   15.154836 &  19.675337&   29.348136  &     31.415853 &  40.855247 \\
\hline $1/16\times 1/16$&      9.5896364 &   15.185968 &  19.723326&   29.477224  &     31.747401 &  41.282853 \\
\hline $1/32\times 1/32$&     9.6205870 &   15.194336 &  19.735243&   29.510335  &     31.855244 &  41.413996 \\
\hline $1/64\times 1/64$&    9.6323169 &   15.196509 &  19.738218&   29.518686  &     31.891900 &  41.454533 \\
\hline Trend & $\nearrow $ & $\nearrow $   & $\nearrow$ & $\nearrow$ &  $\nearrow$& $\nearrow$\\
\hline
\end{tabular}
\end{table}

\begin{table}[htb!]
\centering \caption{Bilinear element for postprocessing method on
$L$ shape domain}\label{Bilinear_LShape}
\begin{tabular}{|c|c|c|c|c|c|c|}\hline
$h$&$\lambda_{1,h}$&$\lambda_{2,h}$&$\lambda_{3,h}$&$\lambda_{4,h}$ &$\lambda_{5,h}$ &$\lambda_{6,h}$ \\
\hline $1/4\times 1/4$&        10.164089 &    15.980053 &   20.773284&  32.476652 &     35.807091 &   48.127411   \\
\hline $1/8\times 1/8$&        9.7907347 &    15.392383 &   19.994161&  30.245451 &     32.945969 &   43.311982 \\
\hline $1/16\times 1/16$&      9.6867567 &    15.246072 &   19.802707&  29.701314 &     32.192812 &   41.959311 \\
\hline $1/32\times 1/32$&      9.6552886 &    15.209476 &   19.755068&  29.566371 &     31.992010 &   41.603765 \\
\hline $1/64\times 1/64$&      9.6451377 &    15.200312 &   19.743173&  29.532700 &     31.936225 &   41.509772 \\
\hline Trend & $\searrow $ & $\searrow $   & $\searrow$ & $\searrow$ &  $\searrow$& $\searrow$\\
\hline
\end{tabular}
\end{table}
\begin{table}[htb!]
\centering \caption{Biquadratic element for postprocessing method on
$L$ shape domain}\label{Biquadratic_LShape}
\begin{tabular}{|c|c|c|c|c|c|c|}\hline
$h$&$\lambda_{1,h}$&$\lambda_{2,h}$&$\lambda_{3,h}$&$\lambda_{4,h}$ &$\lambda_{5,h}$ &$\lambda_{6,h}$ \\
\hline $1/4\times 1/4$&      9.6733499&    15.208409&     19.749420&    29.581107&     32.072149 &    41.784684   \\
\hline $1/8\times 1/8$&      9.6525125&    15.198129&     19.739857&    29.525433&     31.949976 &    41.516052 \\
\hline $1/16\times 1/16$&    9.6447652&    15.197334&     19.739249&    29.521742&     31.925488 &    41.485165 \\
\hline $1/32\times 1/32$&    9.6417221&    15.197261&     19.739211&    29.521499&     31.917575 &    41.478309 \\
\hline $1/64\times 1/64$&    9.6405166&    15.197253&     19.739209&    29.521483&     31.914580 &    41.475981 \\
\hline Trend & $\searrow $ & $\searrow $   & $\searrow$ & $\searrow$ &  $\searrow$& $\searrow$\\
\hline
\end{tabular}
\end{table}

Figure \ref{Error_LShape} shows the errors of the eigenvalue
approximations by ECR and $EQ_1^{\rm rot}$ elements, postprocessing
methods with lowest order (linear and bilinear) and higher order
(quadratic and biquadratic) elements on the $L$ shape domain. Since
we don't know the exact eigenvalues on the $L$ shape domain, we can
only give the errors of $\widehat{\lambda}_{j,h}-\lambda_{j,h}$ and
$\widetilde{\lambda}_{j,h}-\lambda_{j,h}\ (j=1,2,3,4,5,6)$. Since
the eigenfunctions here are singular, the convergence order by
postprocessing with higher order element can not be improved which
is shown in Figure \ref{Error_LShape}. 
\begin{figure}[ht]
\centering
\includegraphics[width=7cm,height=7cm]{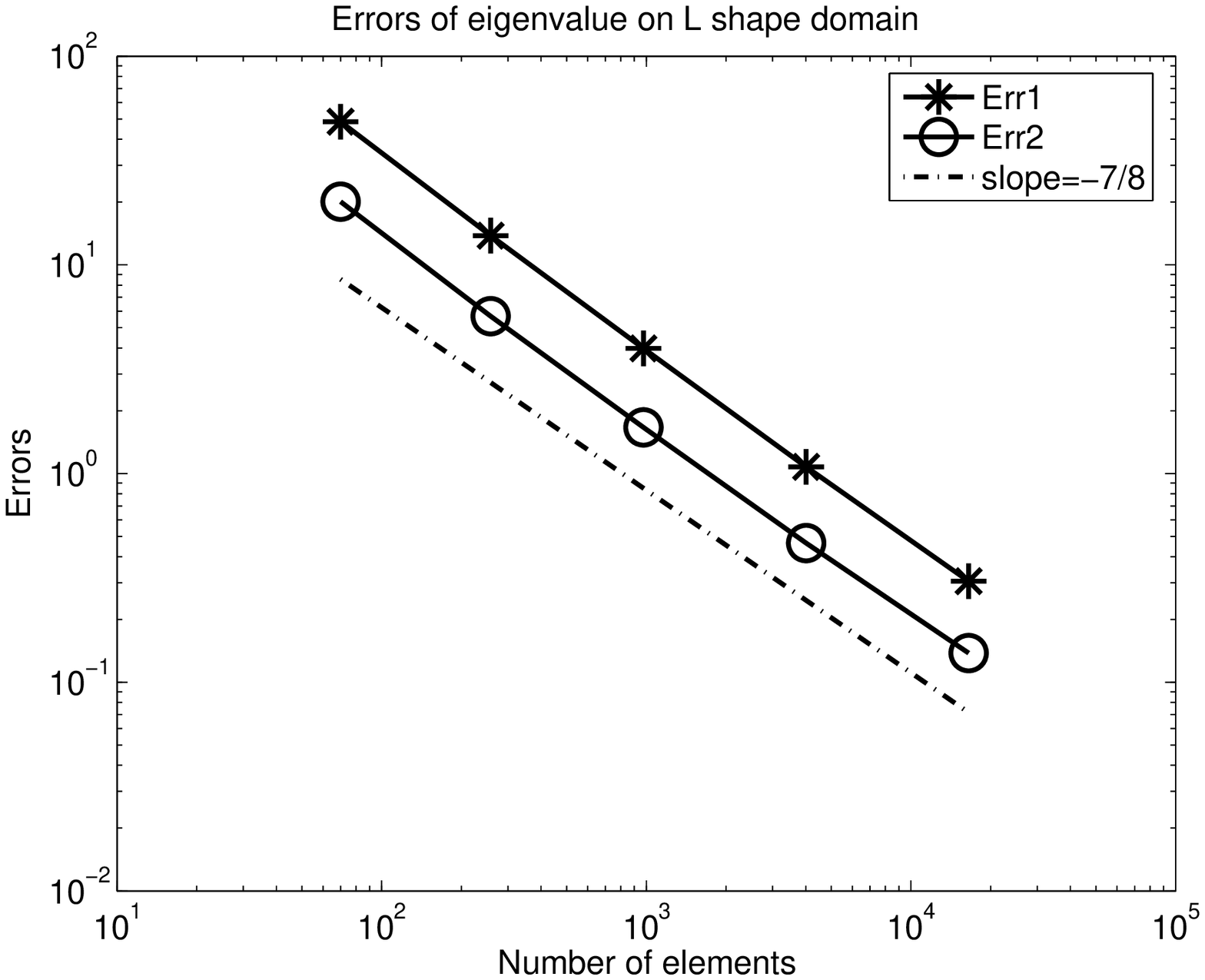}
\includegraphics[width=7cm,height=7cm]{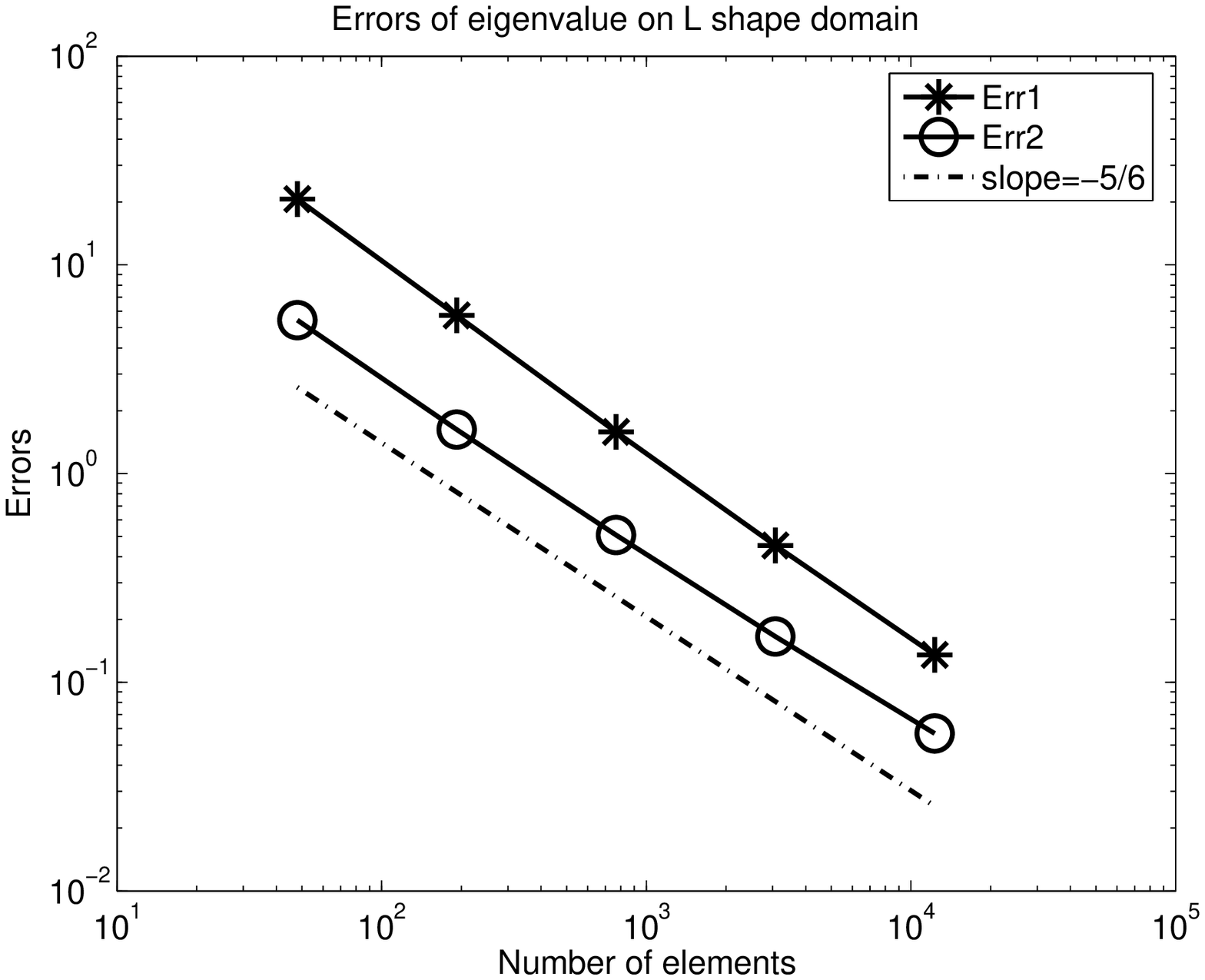}
\caption{The errors for the eigenvalue approximations on $L$ shape
domain by ECR (left) and $EQ_1^{\rm rot}$ (right), where ${\rm
Err1}=\sum_{j=1}^6(\widehat{\lambda}_{j,h}-\lambda_{j,h})$ and ${\rm
Err2}=\sum_{j=1}^6(\widetilde{\lambda}_{j,h}-\lambda_{j,h})$}
\label{Error_LShape}
\end{figure}

\section{Concluding remarks}
In this paper, we analyzed the lower bound approximation of
eigenvalue problem by nonconforming elements (ECR and $EQ_1^{\rm
rot}$) and also two postprocessing methods to obtain the upper bound
of the eigenvalues. Especially, based on the lower bound
approximations, a new postprocessing method which can produce not
only higher order convergence but also upper bound approximation of
the eigenvalues is proposed. This improves the efficiency of solving eigenvalue
problems and obtain the accurate a posteriori error estimates by the
lower and upper bounds of eigenvalues.

We should point out that all the methods and results here can be
easily extended to the three dimension case. We listed some related
space and results. The corresponding $ECR$ element in
$\mathcal{R}^3$ is defined as
\begin{align}
V_h^{NC}:= &\Big\{v\in L^2({\it  \Omega}):v|_{K}\in
{\rm span}\{1,x,y,z,x^2+y^2+z^2\},\int_{\ell} v|_{K_1}{\rm d}s=\int_{\ell} v|_{K_2}{\rm d}s,\nonumber\\
&\hspace{1cm}\text{ when }K_1\cap K_2=\ell,\ \text{and
}\int_{\ell}v|_{K}{\rm d}s=0,  \text{ if } K\cap\partial {\it
\Omega}=\ell\Big\},
\end{align}
where $K, K_1, K_2\in \mathcal{T}_h$.

The corresponding $EQ_1^{\rm rot}$ element in $\mathcal{R}^3$ is
defined as
\begin{align}
V_h^{NC}:= &\Big\{v\in L^2(\Omega):v|_{K}\in
{\rm span}\{1,x,y,z,x^2,y^2,z^2\},\int_{\ell} v|_{K_1}ds=\int_{\ell} v|_{K_2}ds,\nonumber\\
&\hspace{1cm}\text{if }K_1\cap K_2=\ell, \ \text{ and
}\int_{\ell}v|_{K}ds=0, \text{ if} K\cap\partial \Omega=\ell\Big\},
\end{align}
where $K, K_1, K_2\in \mathcal{T}_h$.

These two nonconforming elements can be used in the three
dimensional case to get the lower bounds of eigenvalues and the
corresponding postprocessing methods can also be constructed to
obtain upper bounds of the eigenvalues.

\end{document}